%% file: Extremal_unipotent_representations_for_the_finite_Howe_correspondence.tex
\documentclass[12pt]{amsart}
\input{Preamble.tex}
\begin{document}
\maketitle
\input{Introduction.tex}
\input{Howe_correspondence.tex}

\input{Cuspidal_unipotent_representations.tex}
\input{Correspondence_between_Weyl_groups.tex}
\input{Extremal_unipotent_representations.tex}

\bibliography{Paper1bib}
\bibliographystyle{acm}
\end{document}

%% file: Preamble.tex
\usepackage[english]{babel} 
\usepackage{amsfonts} 
\usepackage{amsmath} 
\usepackage{amsthm} 
\usepackage{amssymb} 
\usepackage{mathrsfs}
\usepackage{mathtools}
\usepackage{url}
\usepackage{verbatim}

\title[Extremal unipotent representations...]{Extremal unipotent representations for the finite Howe correspondence}
\author{Jesua Epequin Chavez}
\thanks{The author thanks the Academy of Mathematics and System Science, the Institut de Math\'ematiques de Jussieu and Fondecyt (subvenci\'on Nro. 220 - 2014 - Fondecyt) for support during the preparation of this work.}

\DeclareMathOperator{\End}{End}

\DeclareMathOperator{\U}{U}
\DeclareMathOperator{\Or}{O}
\DeclareMathOperator{\GL}{GL}
\DeclareMathOperator{\Sp}{Sp}
\DeclareMathOperator{\Irr}{Irr}
\DeclareMathOperator{\Ind}{Ind}

\DeclareMathOperator{\SO}{SO}
\DeclareMathOperator{\sgn}{sgn}
\DeclareMathOperator{\Tr}{Tr}

\DeclareFontEncoding{LS1}{}{}
\DeclareFontSubstitution{LS1}{stix}{m}{n}
\DeclareSymbolFont{symbols2}{LS1}{stixfrak} {m} {n}
\DeclareMathSymbol{\operp}{\mathbin}{symbols2}{"A8}

\newtheorem{thm}{Theorem}

\newtheorem{defi}{Definition}
\newtheorem{rmk}{Remark}
\newtheorem{lem}{Lemma}
\newtheorem{prop}{Proposition}

\newtheorem*{thm*}{Theorem}

%% file: Introduction.tex
\section*{Introduction}
Let $\mathbb{F}_q$ be a finite field with $q$ elements and odd characteristic. Denote the symplectic group $\Sp_{2n}(\mathbb{F}_q)$ by $\Sp_{2n}(q)$. A pair of reductive subgroups of $\Sp_{2n}(q)$, where each one is the centralizer of the other, is called \emph{reductive dual pair}. We study \emph{irreducible} dual pairs (cf. \cite{Kudla}), because these are the building blocks of all the others. One such pair $(G_m,G'_{m'})$, in $\Sp_{2n}(q)$, can be either symplectic-orthogonal $(\Sp_{2m}(q),\Or_{m'}(q))$, unitary $(\U_m(q),\U_{m'}(q))$, or linear $(\GL_m(q),\GL_{m'}(q))$ with $n=mm'$ in all cases.

Roger Howe introduced in \cite{Howe} a correspondence  $\Theta_{m,m'}:\mathscr{R}(G_m)\rightarrow\mathscr{R}(G'_{m'})$ between the categories of complex representations of these subgroups. It is obtained from a particular representation $\omega$ of $\Sp_{2n}(q)$, called the \emph{Weil representation}.

For a fixed irreducible \emph{unipotent} representation $\pi$ of $G_m$ (cf. \cite{Lusztig2}), our main goal is to find certain extremal (i.e.\ minimal and maximal) representations in the set of irreducible components of $\Theta_{m,m'}(\pi)$, for unitary and symplectic-orthogonal pairs. The definition of extremal representation is canonical for unitary pairs, for symplectic-orthogonal pairs it is defined by means of the \emph{Springer correspondence}. Our results generalize those found by Aubert, Kra\'skiewicz, and Przebinda in \cite{AKP}. 
 
In \cite{Gerardin}, G\'erardin introduced Weil representations $\omega^\flat$ of linear, symplectic, and unitary groups over finite fields, this representation agrees with $\omega$ when restricted to symplectic-orthogonal pairs; for unitary pairs, their restrictions differ by a character with values in $\{\pm 1\}$. Therefore, we can replace the study of $\Theta_{m,m'}$ by the study of the correspondence $\Theta^\flat_{m,m'}:\mathscr{R}(G_m)\rightarrow\mathscr{R}(G'_{m'})$, induced by $\omega^\flat$.

An important fact about the Howe correspondence $\Theta^\flat_{m,m'}$ is that it respects irreducible cuspidal unipotent representations in the case of \emph{first occurrence} (see Theorem \ref{SectionCuspidalUnipotent}). The study of these representations is crucial because unipotent representations of finite groups of Lie type $G$ belong to Harish-Chandra series $\Irr(G,M,\delta)$, where $\delta$ is a cuspidal unipotent representation of the Levi subgroup $M$.

For a classical group $G_m$, these series are $\Irr(G_m,G_l\times T,\lambda\otimes 1)$, where $\lambda$ is a cuspidal unipotent representation of $G_l$ ($l<m$), and $T$ is the torus of diagonal matrices (of dimension $(m-l)/2$). In \cite{AMR}, Aubert, Michel, and Rouquier showed that $\Theta^\flat_{m,m'}$ maps this series into the set $\mathscr{R}(G'_{m'}, G'_{l'}\times T',\lambda'\otimes 1)$ of representations spanned by $\Irr(G'_{m'},G'_{l'}\times T',\lambda'\otimes 1)$, where $\lambda'$ is the first occurrence of $\lambda$. 

Due to a result by Howlett and Lehrer \cite{HL}, for type I dual pairs, the Howe correspondence between these Harish-Chandra series can be seen as a correspondence between pairs $(W_r,W_{r'})$ of type $\textbf{B}$ Weyl groups. In \cite{AMR}, Aubert, Michel, and Rouquier found explicit representations $\Omega_{r,r'}$ of $W_r\times W_{r'}$, that yield the Howe correspondence between these pairs of Weyl groups for unitary pairs, and made a conjecture for symplectic-orthogonal pairs. This conjecture was recently proved by Pan in \cite{Pan3} (see Section \ref{sec:Howe-Weyl}).

Irreducible representations of the Weyl group $W_n$ are known to be parametrized by bipartitions of $n$. Let $\chi_{\xi',\eta'}$ denote the representation of $W_{r'}$ corresponding to $(\xi',\eta')$, and $\Theta(\xi',\eta')$ denote the set of bipartitions $(\xi,\eta)$ of $r$, such that $\chi_{\xi,\eta}\otimes\chi_{\xi',\eta'}$ appears in $\Omega_{r,r'}$. We can (in most cases) introduce an order on this set. In Theorems \ref{springer_ordre_so1} to \ref{order_phantom_characters_u2} we establish the following.
 
\begin{thm*}
Extremal representations exist and are unique in $\Theta(\xi',\eta')$. 
\end{thm*}

In a subsequent paper we will generalize this result to arbitrary irreducible representations.

%% file: Howe_correspondence.tex
\section{Howe Correspondence}\label{HoweCorrespondence}
\subsection{Dual pairs}\label{DualPairs}
Let $W$ be a symplectic vector space over $\mathbb{F}_q$, and $\Sp(W)$ its \emph{symplectic group}. By choosing a suitable base we can identify this group to a group of matrices, in this situation we denote it by $\Sp_{2n}(q)$, where $\dim W=2n$. The centralizer in $G$ of a subgroup $H$ will be denoted by $C_G(H)$.
\begin{defi}
A \emph{reductive dual pair} $(G,G')$ in $\Sp(W)$ is a pair of reductive subgroups $G$ and $G'$ of $\Sp(W)$ such that 
$$
C_{\Sp(W)}(G)=G',\hspace{5pt} \mbox{and} \hspace{5pt} C_{\Sp(W)}(G')=G.
$$
We will usually omit the word reductive and call $(G,G')$ a dual pair.
\end{defi}
If $W = W_1\operp W_2$ is an orthogonal sum decomposition, and if $(G_1,G_1')$ and $(G_2,G_2')$ are dual pairs in $\Sp(W_1)$ and $\Sp(W_2)$ respectively, then $(G, G') = (G_1\times G_2, G_1'\times G_2')$ is a dual pair in $\Sp(W)$. A dual pair $(G, G')$ which does not arise in this way is said to be \emph{irreducible}. Every dual pair can be written as a product of irreducible dual pairs. These, in turn, are one of the following. 
 
(1) Let $V_1$ and $V_2$ be vector spaces over $\mathbb{F}_q$. Suppose $V_1$ has a symplectic form  $\langle\hspace{3pt},\hspace{2pt}\rangle_1$, and $V_2$ has a quadratic form $\langle\hspace{3pt},\hspace{2pt}\rangle_2$.  

The $\mathbb{F}_q$-vector space $W=V_1\otimes_{\mathbb{F}_q} V_2$ has a symplectic form defined by
$$
\langle u_1\otimes u_2, v_1\otimes v_2\rangle=\langle u_1,v_1 \rangle_1\langle u_2,v_2 \rangle_2.
$$
We can see $\Sp(V_1)$ and $\Or(V_2)$ as subgroups of $\Sp(W)$ via the natural map $\Sp(V_1)\times \Or(V_2) \rightarrow \Sp(W)$. The irreducible pair $(\Sp(V_1),\Or(V_2))$ so obtained is called \emph{symplectic-orthogonal}.

(2) Consider the quadratic extension $\mathbb{F}_{q^2}$ of $\mathbb{F}_q$ and let $F$ denote its Frobenius morphism. Let $V_1$ (resp. $V_2$) be a vector space over $\mathbb{F}_{q^2}$ with a non-degenerate skew-Hermitian (resp. Hermitian) form $\langle\hspace{3pt},\hspace{2pt}\rangle_1$ (resp. $\langle\hspace{3pt},\hspace{2pt}\rangle_2$), and let $\U(V_1)$ (resp. $\U(V_2)$) be the corresponding unitary group.

The $\mathbb{F}_{q^2}$-vector space $V=V_1\otimes_{\mathbb{F}_{q^2}} V_2$ can be equipped with the skew-hermitian form
$$
\langle u_1\otimes u_2, v_1\otimes v_2\rangle=\langle u_1,v_1 \rangle_1{}^F\langle u_2,v_2 \rangle_2.
$$
Via the natural map $\U(V_1)\times \U(V_2) \rightarrow \U(V)$, we can see $\U(V_1)$ and $\U(V_2)$ as subgroups of $\U(V)$. 

Denote by $W$ the $\mathbb{F}_q$-vector space underlying $V$. Composing the previous form with $\Tr_{\mathbb{F}_{q^2}/\mathbb{F}_q}$ yields a symplectic form on $W$. Moreover, the unitary group $U(V)$ is embedded in $\Sp(W)$. The irreducible dual pair  $(\U(V_1),\U(V_2)$ of subgroups of $\Sp(W)$ is called \emph{unitary}.

Unitary and symplectic-orthogonal pairs are said to be of \emph{type I} dual pairs. 

(3) Let $V_1$ and $V_2$ be vector spaces over $\mathbb{F}_q$. The natural action of $\GL(V_1)\times\GL(V_2)$ on $V=V_1\otimes V_2$, induces an action on its dual $V^*$. By considering the diagonal action we get a map $\GL(V_1)\times\GL(V_2)\rightarrow \GL(W)$, where $W=V\oplus V^*$ is a symplectic vector space with form 
$$
\langle x+x^*,y+y^*\rangle=y^*(x)-x^*(y),
$$
that makes $\GL(V_1)$ and $\GL(V_2)$ subgroups of $\Sp(W)$. Irreducible dual pairs $(\GL(V_1),\GL(V_2))$ arising this way are called \emph{linear} (or of \emph{type II}).

\subsection{Weil representations}
In order to define the Howe correspondence, we must introduce the \emph{Heisenberg group}. This is the group with underlying set $H(W)= W\times\mathbb{F}_q$ and product
$$
(w,t) \cdot (w',t') = (w+w',t+t'+\frac{1}{2}\langle w,w'\rangle)
$$
Let $\rho$ be an irreducible representation of $H(W)$. Its restriction to the center $Z\simeq\mathbb{F}_q$ of $H(W)$ equals $\psi_\rho \cdot 1$, for a certain character $\psi_\rho$ of $\mathbb{F}_q$. 
\begin{thm}\cite{Mackey}\label{StoneVonNeumann}
For any non-trivial character $\psi$ of $Z$ there exists (up to equivalence) a unique irreducible representation $\rho$ of $H(W)$ such that $\psi_\rho=\psi$.
\end{thm}
This representation is known as the \emph{Heisenberg representation}. It depends on $\psi$, so we denote it by $\rho_\psi$. 
 
The natural action of $\Sp(W)$ on $H(W)$ fixes the elements of its center. Hence, for a fixed character $\psi$ of $\mathbb{F}_q$, the representations $\rho_\psi$ and $x\cdot\rho_\psi$ agree on $Z$, for any $x\in \Sp(W)$. Theorem \ref{StoneVonNeumann} implies that there is an operator $\omega_\psi(x)$ verifying
$$
\rho_\psi(x\cdot w,t)=\omega_\psi(x)\rho_\psi(w,t)\omega_\psi(x)^{-1}.
$$
This defines a projective representation $\omega_\psi$ of $\Sp(W)$, which can be lifted to an actual representation of $\Sp(W)$, known as the \emph{Weil representation}.

For an irreducible dual pair $(G,G')$, pulling back this representation by $G\times G'\rightarrow\Sp(W)$, we get a representation $\omega_{G,\; G'}$ of $G\times G'$. It decomposes as :
$$
\omega_{G,\; G'}=\sum_{\pi \in \Irr(G)} \pi\otimes\Theta(\pi),
$$
where $\Theta(\pi)$ is a (not necessarily irreducible) representation of $G'$. This map $\Theta$, from the set of irreducible representations of $G$ to the set of representations of $G'$, is known as the \emph{Howe correspondence}.

For unitary dual pairs we introduce the representation 
$$
\omega_\psi^\flat = \nu_m \otimes \omega_\psi, \hspace{5pt} \mbox{ on } U(V), 
$$
where $\nu_m(u)= (\det u)^{(q+1)/2}$ defines a character of $U(V)$. This representation induces a correspondence denoted by $\Theta_\psi^\flat$.

\subsection{Witt towers}\label{WittTower}
Some nice properties of the Howe correspondence involve its compatibility with \emph{Witt towers} $\mathbf{T}=\{G_n\}_{n\in\mathbb{N}}$ :

(1) For unitary groups there are two, one whose groups are $G_n=U_{2n}(q)$, for $n\in\mathbb{N}$, and the other for groups $G_n=U_{2n+1}(q)$, for $n\in\mathbb{N}$. The first one will be denoted by $\mathbf{U}^+$, and the second one by $\mathbf{U}^-$.

(2) In the symplectic case there is only one, it is formed by groups $G_n=\Sp_{2n}(q)$, for $n\in\mathbb{N}$. It will be denoted by $\mathbf{Sp}$.

(3) Even-orthogonal groups provide two Witt towers whose groups are $G_n=\Or^+_{2n}(q)$, and $G_n=\Or^-_{2n}(q)$, for positive integers $n$. These will be denoted by $\mathbf{O}^+$, and $\mathbf{O}^-$ respectively.
\begin{defi}
The pair $(G_m,G'_{m'})$ is in the \emph{stable range (with $G_m$ smaller)} if the defining module for $G'_{m'}$ has a totally isotropic subspace of dimension greater of equal than the dimension of the defining module of $G_m$. 
\end{defi}
For instance, for pairs $(\Sp_{2m}(q),\Or_{2m'}(q))$ the stable range condition (with $\Or_{2m'}(q)$ smaller) means that $m\geq 2m'$.
\begin{prop}\cite[Propositions 4.3 and 4.5]{Kudla}
In the situation above, for every irreducible representation $\pi$ of $G_m$, $\Theta(\pi)\neq 0$
\end{prop}

%% file: Cuspidal_unipotent_representations.tex
\section{Cuspidal unipotent representations}\label{SectionCuspidalUnipotent}
For type I dual pairs $(G_m,G'_{m'})$, the Howe correspondence defined by the representation $\omega^\flat_{m,m'}$ of $G_m\times G'_{m'}$, will be denoted $\Theta_{m,m'}^\flat$. In this section show the behaviour of this correspondence with respect to Harish-Chandra series of cuspidal unipotent representations. 

The occurrence of a cuspidal irreducible representation $\pi$ of $G_m$ in the Howe correspondence for $(G_m,G'_{m'})$ with $m'$ minimal (that is $\Theta_{m,k}^\flat(\pi)=0$ for $k<m'$), is referred to as the \emph{first occurrence}\index{first occurrence}. The integer $m'$ will be called \emph{first occurrence index}\index{first occurrence!index} for $\pi$. In this case, the representation $\Theta_{m,m'}^\flat(\pi)$ is cuspidal and irreducible (cf. \cite[Theorem 2.2]{Adams-Moy}), and we denote it by $\theta^\flat(\pi)$.

Another important fact about the correspondence $\Theta_{m,m'}^\flat$, is that it respects unipotent representations. If the representation $\pi\otimes\pi'$ of $G_m\times G'_{m'}$ appears in $\omega^\flat_{m,m'}$, then $\pi$ is unipotent if and only if $\pi'$ is unipotent (cf. \cite[Proposition 2.3]{AMR}).

Between groups belonging to a dual pair, the only ones having cuspidal unipotent representations are $\GL_1(q)$, $\Sp_{2k(k+1)}(q)$, $\U_{(k^2+k)/2}(q)$, and $\SO^\epsilon_{2k^2}(q)$ with $\epsilon=\sgn(-1)^k$. Moreover, this representation is unique (trivial for the linear group) (cf. \cite[Theorem 3.22]{Lusztig3}). For the last three groups we denoted it by $\lambda_k$.

Inducing the cuspidal unipotent representation $\lambda_k$ of $\SO^\epsilon_{2k^2}(q)$ to $\Or^\epsilon_{2k^2}(q)$, we obtain two irreducible cuspidal representations, $\lambda^I_k$ and $\lambda_k^{II}$, which differ by tensoring with the $\sgn$ character of $\Or^\epsilon_{2k^2}(q)$.
\begin{thm}\cite[Theorems 4.1 and 5.2]{Adams-Moy}\label{CuspidalUnipotentHowe}
	The Howe correspondence $\Theta_{m,m'}^\flat$ for type I dual pairs takes cuspidal unipotent representations to cuspidal unipotent representations as follows :
	\begin{itemize}
		\item For towers $(\mathbf{Sp},\mathbf{O}^\epsilon)$, $\lambda_k$ corresponds to $\lambda_k^{II}$ if $\epsilon$ is the sign of $(-1)^k$, and to $\lambda_{k+1}^I$ otherwise. 
		\item For towers $(\mathbf{U}^\epsilon,\mathbf{U}^{\epsilon'})$, $\lambda_k$ corresponds to $\lambda_{k'}$, where $k'=k+1$ or $k'=k-1$. We take $k$ so that $\epsilon$ is the sign of $(-1)^{k(k+1)/2}$, and we choose $k'$ such that $\epsilon'=(-1)^{k'(k'+1)/2}$.
	\end{itemize}
	Moreover, these cases give the first occurrence $\theta^\flat(\lambda_k)$ of $\lambda_k$.
\end{thm}
This theorem allows us to write the Howe correspondence $\Theta_{m,m'}^\flat$, between cuspidal unipotent representations, as a function on natural integers $\theta:\mathbb{N}\rightarrow\mathbb{N}$, defined by $\theta^\flat(\lambda_k)=\lambda_{\theta(k)}$.

Let $\mathbf{T}$, and $\mathbf{T'}$ be two Witt towers such that pairs $(G_m,G'_{m'})$ with $G_m \in \mathbf{T}$, and $G'_{m'} \in \mathbf{T'}$ are of type I. For a fixed group $G_m$, a group $G_l$ in the same Witt tower, and such that $l<m$, together with a cuspidal representation $\lambda$ of $G_l$, yield a cuspidal pair $(G_l\times T_{1/2(m-l)},\lambda\otimes 1)$, where $T_k$ denotes the maximal torus of diagonal matrices in $\GL_k$.
\begin{thm}\cite[Th\'eor\`eme 3.7]{AMR}\label{compatibilityHC}
 In the situation above, let $l'$ be the first occurrence index of $\lambda$, and $\lambda'=\theta^\flat(\lambda)$ the corresponding cuspidal representation of $G'_{l'}\in\mathbf{T}'$. For $\gamma \in \Irr(G_m,\lambda\otimes 1)$, $\Theta_{m,m'}^\flat(\gamma)=0$ whenever $m'<l'$, and $\Theta_{m,m'}^\flat(\gamma) \in \mathscr{R}(G'_{m'},\lambda'\otimes 1)$ otherwise. Moreover, the representation $\lambda$ is unipotent if and only if the same holds for $\lambda'$. 
\end{thm}

%% file: Correspondence_between_Weyl_groups.tex
\section{Correspondence between Weyl groups}\label{sec:Howe-Weyl}

Let $\mathbf{G}$ be a reductive group defined over $\mathbb{F}_q$, and $\mathbf{P}=\mathbf{M}\mathbf{U}$ be a Levi decomposition of the rational parabolic subgroup $\mathbf{P}$. For a cuspidal representation $\delta$ of $M$ set
 $$
 W_\mathbf{G}(\delta) = \{x\in N_{G}(\mathbf{M})/M : {}^x\delta=\delta\}.
 $$
\begin{thm}\cite[Corollary 5.4]{HL} and \cite[Corollary 2]{Geck1993}\label{How-Leh}
There is an isomorphism 
$$
\End_G(R_\mathbf{M}^\mathbf{G}(\delta)) \simeq \mathbb{C}[W_\mathbf{G}(\delta)].
$$
In particular, irreducible representations in the \emph{Harish-Chandra series} $\Irr(\mathbf{G},\mathbf{M},\delta)$ are indexed by irreducible representations of $W_\mathbf{G}(\delta)$.
\end{thm}
We refer to this parametrization as the \emph{Howlett-Lehrer bijection}.
 
Let $G_m$ belong to a Witt tower of symplectic, unitary or orthogonal groups. Harish-Chandra series for $G_m$ can be parametrized by cuspidal pairs $(G_{m-\lvert\mathbf{t}\rvert}\times\GL_{\mathbf{t}},\varphi\otimes \boldsymbol{\sigma})$, for partitions $\mathbf{t}=(t_1,\ldots,t_r)$ such that $\lvert\mathbf{t}\rvert\leq m$. Series containing unipotent representations correspond to representations $\varphi\otimes\boldsymbol{\sigma}$, cuspidal \emph{and} unipotent. This forces the cuspidal pair to become $(G_{m(k)}\times T_r,\lambda_k\otimes 1)$, where $\lambda_k$ is a cuspidal unipotent representation of $G_{m(k)}$, $r=(m-m(k))/2$, and $m(k)$ is equal to $k^2+k$ for symplectic, $(k^2+k)/2$ for unitary, and $k^2$ for orthogonal groups. For the first two kinds of groups, the Harish-Chandra series will be denoted by $\Irr(G_m)_k$, and the set of representations spanned by this series will be denoted by $\mathscr{R}(G_m)_k$. For orthogonal groups $\Or^\epsilon_{2m}(q)$ and $k$ verifying $\epsilon=(-1)^k$, the Harish-Chandra series corresponding  $\lambda_k^I$ and $\lambda_k^{II}$ are denoted by $\Irr(\Or^\epsilon_{2m}(q))_k^I$, and $\Irr(\Or^\epsilon_{2m}(q))_k^{II}$, and their spanned sets by $\mathscr{R}(\Or^\epsilon_{2m}(q))_k^I$, and $\mathscr{R}(\Or^\epsilon_{2m}(q))_k^{II}$ respectively. 

For symplectic or unitary groups, the unicity of $\lambda_k$ implies that the condition on the elements of the group $W_{G_m}(1\otimes\lambda_k)$ is trivial. Therefore, the Howlett-Lehrer bijection identifies $\Irr(G_m)_k$ to $N_{G_m}(L_k)/L_k$, which is a Weyl group of type $\mathbf{B}_{(m-m(k))/2}$. The same reasoning allows us to state that the series $\Irr(\Or^\epsilon_{2m}(q))_{k}^{I}$ and $\Irr(\Or^\epsilon_{2m}(q))_{k}^{II}$ are in bijection with the irreducible representations of a Weyl group of type $\mathbf{B}_{m-k^2}$.

These remarks, together with Theorems \ref{CuspidalUnipotentHowe} and \ref{compatibilityHC} imply that for type I pairs $(\Sp_{2m}(q),\Or^\epsilon_{2m'}(q))$, and $(\U_m(q),\U_{m'}(q))$ the Howe correspondence between Harish-Chandra series of cuspidal unipotent representations leads to a correspondence between pairs of type $\mathbf{B}$ Weyl groups : $(\mathbf{B}_{m-k(k+1)},\mathbf{B}_{m'-\theta(k)^2})$ for symplectic-orthogonal pairs, and $(\mathbf{B}_{\frac{1}{2}(m-k(k+1)/2)},\mathbf{B}_{\frac{1}{2}(m'-\theta(k)(\theta(k)+1)/2)})$ for unitary pairs.  

Let $(G_m,G'_{m'})$ be a type I dual pair and $(W_r,W_{r'})$ be one of the corresponding pairs of Weyl groups from the previous paragraph. Denote $\theta(k)$ by $k'$, and denote the  projection of $\omega_{m,m'}$ onto $\mathscr{R}(G_m)_k\otimes \mathscr{R}(G'_{m'})_{k'}$ by $\omega_{m,m',k}$.
\begin{thm}\cite[Theorem 3.31]{Pan3}\label{ConjectureSymplectic-Orthogonal}
 For the symplectic-orthogonal dual pair $(\Sp_{2m}(q),\Or^\epsilon_{2m'}(q))$, there is a bijection
$$
\Irr(\Sp_{2m}(q))_k\times \Irr(\Or^\epsilon_{2m'}(q))^{\Gamma}_{k'}\simeq \Irr(W_r\times W_{r'}),
$$
where $\Gamma=II$ if $\epsilon=(-1)^k$ and $\Gamma= I$ otherwise. Moreover, it identifies $\omega_{m,m',k}$ to the representation $\Omega_{r,r'}$ whose character is :
\begin{align}\label{so1}
\sum_{l=0}^{\min(r,r')}\sum_{\chi\in\Irr(W_l)}(\Ind_{W_l\times W_{r-l}}^{W_r}\chi \otimes \sgn) \otimes (\Ind_{W_l\times W_{r'-l}}^{W_r'}\chi \otimes \sgn),
\end{align} 
for $(\Sp_{2m}(q),\Or^\epsilon_{2m'}(q))$ if $\epsilon=(-1)^k$; and
\begin{align}\label{so2}
\sum_{l=0}^{\min(r,r')}\sum_{\chi\in\Irr(W_l)}(\Ind_{W_l\times W_{r-l}}^{W_r}\chi \otimes 1) \otimes (\Ind_{W_l\times W_{r'-l}}^{W_r'}\chi \otimes \sgn),
\end{align}
otherwise.
\end{thm}
\begin{thm}\cite[Theorem 3.10]{AMR}\label{ReductionWeilUnipotent}
Let $(U_m(q),U_{m'}(q))$ be a unitary dual pair. The bijection
$$
\Irr(U_m(q))_k\times \Irr(U_{m'}(q))_{k'}\simeq \Irr(W_r\times W_{r'}),
$$
identifies the representation $\omega_{m,m',k}$ with the representation $\Omega_{r,r'}$ whose character is :
\begin{align}\label{u1}
\sum_{l=0}^{\min(r,r')}\sum_{\chi\in\Irr(W_l)}(\Ind_{W_l\times W_{r-l}}^{W_r}\chi \otimes 1) \otimes (\Ind_{W_l\times W_{r'-l}}^{W_r'}\sgn\chi \otimes 1),
\end{align}
for the pair $(\U_m(q),\U_{m'}(q))$, if $k$ is odd or $k=k'=0$; and
\begin{align}\label{u2}
\sum_{l=0}^{\min(r,r')}\sum_{\chi\in\Irr(W_l)}(\Ind_{W_l\times W_{r-l}}^{W_r}\chi \otimes \sgn) \otimes (\Ind_{W_l\times W_{r'-l}}^{W_r'}\sgn\chi \otimes 1),
\end{align}
otherwise.
\end{thm}  

%% file: Extremal_unipotent_representations.tex
\section{Extremal unipotent representations}\label{Extremal_unipotent}
We present symplectic-orthogonal and unitary pairs separately, because the definiton of ``extremal" (i.e.\ ``minimal" and ``maximal") representation changes from one pair to the other. 

There is a natural order on the set of partitions of an integer $n\in\mathbb{N}$. Take $\mu$ and $\mu'$ two partitions of the same integer. Then, $\mu\leq\mu'$ if and only if 
$$
\mu_1+\ldots+\mu_k\leq \mu'_1+\ldots+\mu'_k,\hspace{5pt} \mbox{for all } k\in\mathbb{N}.
$$ 
Following \cite{AMR} we introduce another order between partitions.
\begin{defi}
Let $\mu$ and $\mu'$ be partitions (of possibly different integers). We denote 
$$
\mu\preceq\mu' \mbox{ if and only if } \mu'_{i+1}\leq \mu_i\leq \mu'_i.
$$
\end{defi} 
This defines an order on partitions. It says that $\mu\preceq\mu'$ if the Young diagram of $\mu$ is contained in the one of $\mu'$ and that we can go from the first to the second by adding at most one box per column.

We denote by $\mathscr{P}_2(n)$ the set of bipartitions $(\lambda,\mu)$ of $n$. Irreducible characters of a Weyl group $W_n$ of type $B$ or $C$ are known to be parametrised by bipartitions of $n$ (cf. \cite[Theorem 5.5.6]{Geck-Pfeiffer}). We denote by $\chi_{\lambda,\mu}$ the irreducible representation of $W_n$, corresponding to the bipartition $(\lambda,\mu)$ of $n$.
 \begin{prop}\cite[Chapter 5]{Geck-Pfeiffer}\label{Weyl_representation_tensor}
Let $(\lambda,\mu)$ be a bipartition of the integer $r$, then
\begin{itemize}
\item[1.] $\Ind_{W_r\times W_{l-r}}^{W_l}\chi_{\lambda,\mu}\otimes\sgn = \sum_{\mu\preceq\mu'}\chi_{\lambda,\mu'}.$
\item[2.] $\Ind_{W_r\times W_{l-r}}^{W_l}\chi_{\lambda,\mu}\otimes 1 = \sum_{\lambda\preceq\lambda'}\chi_{\lambda',\mu}.$
\item[3.] $\sgn \otimes \chi_{\lambda,\mu}=\chi_{\mu,\lambda}$
\end{itemize}
\end{prop}
Achar and Henderson \cite{Achar-Henderson} introduced the following order between bipartitions.
\begin{defi}\label{lexicographical_order}
 For $(\rho,\sigma)$, $(\mu,\nu)\in\mathscr{P}_2(n)$ we say that $(\rho,\sigma)\leq(\mu,\nu)$, if and only if, the following inequalities hold for all $k\geq 0$ :
\begin{align*}
\rho_1+\sigma_1+\cdots+\rho_k+\sigma_k & \leq\mu_1+\nu_1+\cdots+\mu_k+\nu_k, \mbox{ and } \\
\rho_1+\sigma_1+\cdots+\rho_k+\sigma_k+\rho_{k+1} & \leq\mu_1+\nu_1+\cdots+\mu_k+\nu_k+\mu_{k+1}.
\end{align*}
We will refer to this as the \emph{natural order}.
\end{defi}
The following technical result will be used for both unitary and symplectic-orthogonal pairs.
\begin{lem}\label{TechnicalLemmaSOandUpairs}
Let $(\xi,\eta)$, and $(\xi',\eta')$ be bipartitions of $l$, and $l'$ respectively. Suppose the number of parts of $\eta$ and $\eta'$ are equal, and that $\xi$ has one more part than $\xi'$ (this can be achieved by adding zeroes). Call $t_0$ (resp. $t_1$) the number of parts of $\eta'$ (resp. $\xi'$). Finally suppose that $\xi'\preceq\xi$, and $\eta\preceq\eta'$.
\begin{itemize}
\item[1.]  For $P\subset \{1,\ldots,t_0\}$, and $Q\subset \{1,\ldots,t_1\}$,
$$
\xi_1+\sum_Q\xi_{i+1}+\sum_P\eta_i\geq (l-l') +\sum_Q\xi'_i+\sum_P\eta'_i.
$$
\item[2.] For $P\subset \{1,\ldots,t_0\}$, and $Q\subset \{1,\ldots,t_1+1\}$,
$$
\sum_Q\xi_i+\sum_P\eta_i\leq l-l' +\eta'_1+ \sum_Q\xi'_i + \sum_P\eta'_{i+1},
$$
where we set $\xi_{t_1+1}=\eta'_{t_0+1}=0$.
\end{itemize}
\end{lem}   
\begin{proof}
By definition,
\begin{align}\label{eq_small_so2}
\xi_{i+1}\leq \xi'_i\leq\xi_i, \mbox{ and } \eta_j\leq\eta'_j
\end{align}
for all $i=1,\ldots,t_1$, and $j=1\ldots,t_0$. Rewriting $\lvert\xi\rvert+\lvert\eta\rvert - \lvert\xi'\rvert-\lvert\eta'\rvert = l-l'$ as
$$
\xi_1+\sum_{i=1}^{t_1}(\xi_{i+1}-\xi'_i)+\sum_{i=1}^{t_0}(\eta_i-\eta'_i)=l-l',
$$ 
inequalities (\ref{eq_small_so2}) imply
\begin{align*}
\xi_1+\sum_Q(\xi_{i+1}-\xi'_i)+\sum_P(\eta_i-\eta'_i)\geq l-l',
\end{align*}
for $P\subset \{1,\ldots,t_0\}$, and $Q\subset \{1,\ldots,t_1\}$ arbitrary. The proof of item (2) is similar.
\end{proof}

\subsection{Symplectic-orthogonal pairs}
Let $W_n=W(C_n)$ be the Weyl group of $\Sp_{2n}(\overline{\mathbb{F}}_q)$. In \cite{Lusztig} Lusztig extended the Springer correspondence (introduced in \cite{Springer}) to finite fields of large characteristic. This correspondence is an injective map from the set of irreducible representations of $W_n$, into the set of pairs $(\mathcal{O},\psi)$, where $\mathcal{O}$ is a unipotent conjugacy class of $\Sp_{2n}(\overline{\mathbb{F}}_q)$, and $\psi$ is an irreducible character of the group $A(u)$ of connected components of the centraliser $C(u)$ of any $u\in\mathcal{O}$. 
 
Recall that a partition is called \emph{symplectic} if each odd part appears with even multiplicity. There is a bijection between symplectic partitions of $2n$ and unipotent conjugacy classes of $\Sp_{2n}(\overline{\mathbb{F}}_q)$. We denote by $\mathcal{O}_\lambda$ the unipotent orbit associated to the symplectic partition $\lambda$.  
 
Let $\lambda$ be a symplectic partition of $2n$. By adding a zero if necessary, we can suppose $\lambda$ has an even number $2k$ of parts. We now set $\lambda_j^*=\lambda_{2k-j+1}+j-1$, for $j=1,\ldots,2k$. We divide $\lambda^*$ into its odd and even parts. Is has the same number of each. Let the odd parts be
$$
2\xi_1^*+1<2\xi_2^*+1<\ldots<2\xi_k^*+1
$$
and the even parts be
$$
2\eta_1^*<2\eta_2^*<\ldots<2\eta_k^*.
$$ 
Set $\xi_i=\xi_{k-i+1}^*-(k-i)$, and $\eta_i=\eta_{k-i+1}^*-(k-i)$ for each $i$. We obtain in this way a bipartition $(\xi,\eta)=(\xi(\lambda),\eta(\lambda))$ of $n$. The injective map, sending $\lambda$ to $(\xi(\lambda),\eta(\lambda))$, is closely related to the Springer correspondence. 
 
Conversely, let $(\xi,\eta)$ be a bipartition of $n$, we  ensure that $\xi$ has one more part than $\eta$ by adding zeroes to $\xi$ if necessary. Let $k$ the number of parts of $\eta$. We associate to $(\xi,\eta)$ the following \emph{$u$-symbol}
$$
\begin{pmatrix}
\xi_{k+1} \hspace{0.3cm} \xi_k + 2 \hspace{0.2cm} \cdots  \hspace{0.2cm} \xi_1 + 2k\\
\hspace{0.2cm} \eta_k +1 \hspace{0.2cm} \cdots \hspace{0.2cm} \eta_1 +2k-1
\end{pmatrix}.
$$ 
The bipartition $(\xi,\eta)$ is in the image of the above map if and only if its associated $u$-symbol is \emph{distinguished}, that is
$$
\xi_{k+1} \leq \eta_k +1 \leq \xi_k +2 \leq \eta_{k-1} +3\leq\cdots.
$$
In this situation the Springer map sends the representation $\chi_{\xi,\eta}$ of $W_n$ to the pair $(\mathcal{O}_\lambda, 1)$ where $\lambda$ is the symplectic partition such that $(\xi,\eta)=(\xi(\lambda),\eta(\lambda))$, and $1$ is the trivial representation of $A(u)$. 

The set of all $u$-symbols which share the same entries with the same multiplicities (in different arrangements) is called a \emph{similarity class}. Each similarity class contains exactly one distinguished $u$-symbol. 

Suppose the bipartition $(\xi,\eta)$ is not in the image of the above map, and let $(\xi',\eta')$ be the bipartition corresponding to the distinguished $u$-symbol that is similar to the one of $(\xi,\eta)$. Let $\lambda$ be the symplectic partition verifying $(\xi',\eta')=(\xi(\lambda),\eta(\lambda))$, then the Springer correspondence maps $\chi_{\xi,\eta}$ into the pair $(\mathcal{O}_\lambda, \psi)$ for some character $\psi$ of $A(u)$. In this situation we write $\lambda=\lambda(\xi,\eta)$. This defines a surjective map sending bipartitions of $n$, to symplectic partitions of $2n$.
\begin{rmk}\label{ClosureOrder}
The \emph{closure order} between unipotent conjugacy classes is compatible with the (natural) order on symplectic partitions.
\end{rmk}
For the symplectic-orthogonal pair $(\Sp_{2m}(q),\Or^\epsilon_{2m'}(q))$, the Howe correspondence between unipotent Harish-Chandra series induces a correspondence for the pair of type $\mathbf{B}$ Weyl groups $(W_{m-k(k+1)}, W_{m'-\theta(k)^2})$, where $\theta(k)$ is equal to $k$ or $k+1$ (see Theorem \ref{CuspidalUnipotentHowe}). 

Let $l=m-k(k+1)$ and $l'=m'-\theta(k)^2$. Theorem \ref{ConjectureSymplectic-Orthogonal} asserts that the correspondence between Weyl groups is given by a certain representation $\Omega_{l,l'}$. 
\begin{defi}\label{def_min_so}
Fix a bipartition $(\xi',\eta')$ of $l'$. 
\begin{itemize}
\item[1.] We denote by $\Theta(\xi',\eta')$ the set of all bipartitions $(\xi,\eta)$ of $l$ such that $\chi_{\xi,\eta}\otimes\chi_{\xi',\eta'}$ is an irreducible component of $\Omega_{l,l'}$.
\item[2.]  A \emph{minimal representation} $(\xi_{\mathrm{min}},\eta_{\mathrm{min}})$  in $\Theta(\xi',\eta')$, is the one verifying $\lambda(\xi_{\mathrm{min}},\eta_{\mathrm{min}})\leq \lambda(\xi,\eta)$ for all $(\xi,\eta)$ in $\Theta(\xi',\eta')$. The definition of \emph{maximal representation} is similar.
\end{itemize}
\end{defi}
Note that a priori, these extremal representations need not exist and even if they do, they need not be unique. Our goal is to prove the existence and uniqueness of such representations. 

The explicit form of $\Omega_{l,l'}$ depends on the sign of $\epsilon$, and the parity of $k$. Theorem \ref{ConjectureSymplectic-Orthogonal} presented two cases, we study these independently.

We suppose from now on that $m\geq 2m'$, i.e.\ that the dual pair $(\Sp_{2m}(q),\Or^\epsilon_{2m'}(q))$ is in the stable range (with $\Or^\epsilon_{2m'}(q)$ smaller).  

\subsubsection{First case}
Consider the dual pair $(\Sp_{2m}(q),\Or^+_{2m'}(q))$ for $k$ even, or $(\Sp_{2m}(q),\Or^-_{2m'}(q))$ for $k$ odd. In these cases, Proposition \ref{Weyl_representation_tensor} allows us to rewrite the representation $\Omega_{l,l'}$, given in Theorem \ref{ConjectureSymplectic-Orthogonal} as
\begin{align}\label{so1bis}
\Omega_{l,l'} = \sum_{r=0}^{\min(l,l')}\sum_{(\xi,\zeta)\in\mathscr{P}_2(r)}\sum_{\eta,\eta'}\chi_{\xi,\eta}\otimes\chi_{\xi,\eta'},
\end{align}
where the third sum is over partitions $\eta$ and  $\eta'$ of $l-\lvert\xi\rvert$ and $l'-\lvert\xi\rvert$ such that $\zeta\preceq\eta$ and $\zeta\preceq\eta'$. 
\begin{lem}\label{small_so1}
Let $(\xi',\eta')$ be a bipartition of $l'$.
\begin{itemize}
\item[1.] The bipartition $(\xi,\eta)$ belongs to $\Theta(\xi',\eta')$ if and only if $\eta$ and $\eta'$ have a common predecessor for $\preceq$, and $\xi=\xi'$.
\item[2.] The smallest element of $\Theta(\xi',\eta')$ for the natural order corresponds to $(\xi',(l-l')\cup\eta')$
\item[3.] The largest element of $\Theta(\xi',\eta')$ for the natural order corresponds to $(\xi', (l-l'+\eta'_1+\eta'_2,\eta'_3,\cdots,\eta'_r))$.
\end{itemize}
\end{lem} 
\begin{proof}
Item $1$ is an easy consequence of equality (\ref{so1bis}).

The representations belonging to $\Theta(\xi',\eta')$ correspond to bipartitions having $\xi'$ as first component, and whose second component $\eta$ shares a common predecessor (for $\preceq$) with $\eta'$. Thus, it is enough to prove that the smallest partition (for the natural order) verifying this property, is $(l-l')\cup\eta'$. The number $l-l'$ is the biggest element in this partition, because the stable range condition $m\geq 2m'$ implies that $l\geq 2l'$.
 
Let $\eta'=(\eta'_1,\ldots,\eta'_r)$, $\eta=(\eta_1,\ldots,\eta_{r+1})$, and $\zeta=(\zeta_1,\ldots,\zeta_r)$ be such that $\eta'$, and $\eta$ have $\zeta$ as common predecessor for $\preceq$. By adding zeros, we can suppose that $\eta$ has one more part than $\zeta$ and $\eta'$). By definition
$$
\eta'_{k+1} \leq \zeta_k \leq \eta'_k, \mbox{, and }\eta_{k+1} \leq \zeta_k \leq \eta_k,
$$
for $k=1,\ldots,r$. Hence, 
\begin{align}\label{ineq_small_so1}
\eta'_k\geq\eta_{k+1} \mbox{ for } k=1,\ldots,r.
\end{align}
As $(\xi,\eta)$ and $(\xi,\eta')$ are bipartitions of $l$ and $l'$ respectively, $\lvert\eta\rvert-\lvert\eta'\rvert=l-l'$, i.e.\
\begin{align}\label{eq_small_so1}
\sum_{i=1}^{r+1}\eta_i &= l-l' + \sum_{i=1}^r \eta'_i. 
\end{align}
This equality and the inequalities in (\ref{ineq_small_so1}) provide
$$
\sum_{i=1}^{k+1}\eta_i \geq l-l' + \sum_{i=1}^k \eta'_i, 
$$
for $k=0,\ldots,r$, i.e.\ $\eta\geq (l-l')\cup\eta'$. This proves item $2$. The proof of item $3$ is analogous. 
\end{proof}
The following result states that the map $\lambda=\lambda(\xi,\eta)$ is increasing when restricted to $\Theta(\xi',\eta')$. 
\begin{prop}\label{prop_springer_ordre_so1} Suppose that $(\xi',\zeta)$ and $(\xi',\zeta')$ belong to $\Theta(\xi',\eta')$. Then, $(\xi',\zeta)\leq(\xi',\zeta')$ if and only if $\lambda(\xi',\zeta)\leq \lambda(\xi',\zeta')$. 
\end{prop}
\begin{proof}
By adding zeros we can suppose that the $\zeta$ and $\zeta'$ have the same number $k$ of parts, and that $\xi'$ has one more part than both. Let   
$$
\begin{pmatrix}
\xi'_{k+1} \hspace{0.3cm} \xi'_k + 2 \hspace{0.2cm} \cdots  \hspace{0.2cm} \xi'_1 + 2k\\
\hspace{0.2cm} \zeta_k +1 \hspace{0.2cm} \cdots \hspace{0.2cm} \zeta_1 +2k-1
\end{pmatrix}
\mbox{, and }
\begin{pmatrix}
\xi'_{k+1} \hspace{0.3cm} \xi'_k + 2 \hspace{0.2cm} \cdots  \hspace{0.2cm} \xi'_1 + 2k\\
\hspace{0.2cm} \zeta'_k +1 \hspace{0.2cm} \cdots \hspace{0.2cm} \zeta'_1 +2k-1
\end{pmatrix}
$$
be the $u$-symbols corresponding to $(\xi',\zeta)$, and $(\xi',\zeta')$ respectively. Let
$$
\begin{pmatrix}
\gamma_{2k+1} \hspace{0.3cm} \gamma_{2k-1}  \hspace{0.2cm} \cdots  \hspace{0.2cm} \gamma_1\\
\hspace{0.2cm} \gamma_{2k} \hspace{0.2cm} \cdots \hspace{0.2cm} \gamma_2
\end{pmatrix}
\mbox{, and }
\begin{pmatrix}
\gamma'_{2k+1} \hspace{0.3cm} \gamma'_{2k-1}  \hspace{0.2cm} \cdots  \hspace{0.2cm} \gamma'_1\\
\hspace{0.2cm} \gamma'_{2k} \hspace{0.2cm} \cdots \hspace{0.2cm} \gamma'_2
\end{pmatrix}
$$
be their associated distinguished $u$-symbols. It is easy to see that $\lambda(\xi',\zeta)\leq \lambda(\xi',\zeta')$, is equivalent to $\gamma\leq\gamma'$, where $\gamma=(\gamma_1,\ldots,\gamma_{2k+1})$ and $\gamma=(\gamma_1,\ldots,\gamma_{2k+1})$. 

Supposse that $(\xi',\zeta)\leq(\xi',\zeta')$, or equivalently, that $\zeta\leq\zeta'$. We must verify that $\gamma\leq\gamma'$, i.e.\
$$
\gamma_1+\cdots+\gamma_r \leq \gamma'_1+\cdots+\gamma'_r.
$$
for $r=1,\ldots,ék+1$. The sum to the left above can be written as
$$
\sum_{i=1}^t\zeta_i+2(k-i)+1+\sum_{i=1}^s\xi'_i+2(k-i+1),
$$
for some $t$, and $s$ such that $t+s=r$. The relation $\zeta\leq\zeta'$ implies that the sum above is lower or equal than
$$
\sum_{i=1}^t\zeta'_i+2(k-i)+1+\sum_{i=1}^s\xi'_i+2(k-i+1).
$$
This last sum in lower or equal than $\gamma'_1+\cdots+\gamma'_r$ (because of distinguishedness).

Conversely, supposse that $\gamma\leq\gamma'$. Let $t'=1,\ldots,k$, and choose $r$ minimal such that $\zeta'_1+2(k-1)+1,\ldots,\zeta'_{t'}+2(k-t')+1$ appear in the list $\gamma'_1,\ldots,\gamma'_r$. In this case 
\begin{align}\label{gammaprime}
\gamma'_1+\cdots+\gamma'_r = \sum_{i=1}^{t'}\zeta'_i+2(k-i)+1+\sum_{i=1}^{s'}\xi'_i+2(k-i+1),
\end{align}
where $s'=r-t'$. The previous sum is greater or equal than $\gamma_1+\cdots+\gamma_r$, which in turn can be written as 
\begin{align}\label{gamma}
\gamma_1+\cdots+\gamma_r = \sum_{i=1}^t\zeta_i+2(k-i)+1+\sum_{i=1}^s\xi'_i+2(k-i+1),
\end{align}
for certain $t$ and $s$ verifying $t+s=r$. Suppose that $s\geq s'$. Distinguishedness implies that $\xi'_{s'+i}+2(k+1-s'-i) \geq \zeta_{t+i}+2(k-t-i)+1$, for $i=1,\ldots,s-s'$. Therefore,
\begin{align*}
\sum_{i=1}^{t'}\zeta'_i+2(k-i) & +1+\sum_{i=1}^{s'}\xi'_i+2(k-i+1) \\
& \geq \sum_{i=1}^t\zeta_i+2(k-i)+1+\sum_{i=1}^s\xi'_i+2(k-i+1)\\
& \geq \sum_{i=1}^{t'}\zeta_i+2(k-i)+1+\sum_{i=1}^{s'}\xi'_i+2(k-i+1).
\end{align*}
This implies that $\zeta'_1+\cdots+\zeta'_{t'}\geq\zeta_1+\cdots+\zeta_{t'}$. 
\end{proof}
\begin{thm}\label{springer_ordre_so1}
Let $(\xi',\eta')$ be a bipartition of $l'$. 
\begin{itemize}
\item[1.] There exists a unique minimal representation in $\Theta(\xi',\eta')$, it is given by $(\xi',(l-l')\cup\eta')$.
\item[2.] There exists a unique maximal representation in $\Theta(\xi',\eta')$, it is given by $(\xi', (l-l'+\eta'_1+\eta'_2,\eta'_3,\cdots,\eta'_r))$.
\end{itemize} 
\end{thm}
\begin{proof}
Existence comes from Lemma \ref{small_so1} and Proposition \ref{prop_springer_ordre_so1}. Unicity also comes from Proposition \ref{prop_springer_ordre_so1}. Indeed, this proposition implies that the minimal (resp. maximal) representation in $\Theta(\xi',\eta')$ (see Definition \ref{def_min_so}) is also minimal (resp. maximal) for the natural order.
\end{proof}

\subsubsection{Second case}
We now analyse pairs $(\Sp_{2m}(q),\Or^+_{2m'}(q))$ for $k$ odd, and $(\Sp_{2m}(q),\Or^-_{2m'}(q))$ for $k$ even. Again, Proposition \ref{Weyl_representation_tensor} lets us rewrite the character $\Omega_{l,l'}$, given in Theorem \ref{ConjectureSymplectic-Orthogonal}, as
\begin{align}\label{so2bis}
\Omega_{l,l'} = \sum_{r=0}^{\min(l,l')}\sum_{(\xi,\eta)\in\mathcal{P}_2(r)}\sum_{\xi',\eta'}\chi_{\xi',\eta}\otimes\chi_{\xi,\eta'},
\end{align}
where the sum is over partitions $\xi'$ and $\eta'$ of $l-\lvert\eta\rvert$ and $l'-\lvert\xi\rvert$ such that $\xi\preceq\xi'$ and $\eta\preceq\eta'$.
\begin{lem}\label{so2set}
Let $(\xi',\eta')$ be a bipartition of $l'$.
\begin{itemize}
\item[1.] A bipartition $(\xi,\eta)$ belongs to $\Theta(\xi',\eta')$ if and only if $\xi'\preceq \xi$ and $\eta\preceq\eta'$.
\item[2.] The smallest element of $\Theta(\xi',\eta')$ for natural order is $((l-l')\cup\xi',\eta')$.
\item[3.] The largest element of $\Theta(\xi',\eta')$ for the natural order corresponds to the bipartition $((l-l'+\eta'_1+\xi'_1,\xi'_2,\ldots,\xi'_r),(\eta'_2,\ldots,\eta'_r))$.
\end{itemize} 
\end{lem} 
\begin{proof}
Item 1 is a straightforward consequence of equality (\ref{so2bis}). Items 2 and 3 are a corollary of Lemma \ref{TechnicalLemmaSOandUpairs}.
\end{proof}
Unlike the previous section, in this case the map $\lambda=\lambda(\xi,\eta)$ is not increasing when restricted to $\Theta(\xi',\eta')$. However, we can prove that the representations obtained in items (2) and (3) of the previous theorem, are indeed the minimal and maximal representations in $\Theta(\xi',\eta')$. 
\begin{thm}\label{springer_ordre_so2}
Let $(\xi',\eta')$ be a bipartition of $l'$.
\begin{itemize}
\item[1.] There exists a unique minimal representation in $\Theta(\xi',\eta')$, it corresponds to the bipartition $((l-l')\cup\xi',\eta')$
\item[2.]There exists a unique maximal representation in $\Theta(\xi',\eta')$, it is the bipartition $((l-l'+\eta'_1+\xi'_1,\xi'_2,\ldots,\xi'_r),(\eta'_2,\ldots,\eta'_r))$.
\end{itemize} 
\end{thm}
\begin{proof}
As in the proof of Proposition \ref{springer_ordre_so1}, we can suppose that $\eta$, $\eta'$ and $\xi'$ have the same number $k$ of parts and that $\xi$ has $k+1$ parts. Let
$$
\begin{pmatrix}
\xi_{k+1} \hspace{0.3cm} \xi_k+2  \hspace{0.2cm} \cdots  \hspace{0.2cm} \xi_1+2k\\
\hspace{0.2cm} \eta_k+1 \hspace{0.2cm} \cdots \hspace{0.2cm} \eta_1+2k-1
\end{pmatrix}
\mbox{, and }
\begin{pmatrix}
\xi'_k \hspace{0.3cm} \xi'_{k-1}+2  \hspace{0.2cm} \cdots  \hspace{0.2cm} l-l'+2k\\
\hspace{0.2cm} \eta'_k+1 \hspace{0.2cm} \cdots \hspace{0.2cm} \eta'_1+2k-1
\end{pmatrix}
$$
be the $u$-symbols  corresponding to $(\xi,\eta)$ and $((l-l')\cup\xi',\eta')$ respectively. Let
$$
\begin{pmatrix}
\gamma_{2k+1} \hspace{0.3cm} \gamma_{2k-1}  \hspace{0.2cm} \cdots  \hspace{0.2cm} \gamma_1\\
\hspace{0.2cm} \gamma_{2k} \hspace{0.2cm} \cdots \hspace{0.2cm} \gamma_2
\end{pmatrix}
\mbox{, and }
\begin{pmatrix}
\gamma'_{2k+1} \hspace{0.3cm} \gamma'_{2k-1}  \hspace{0.2cm} \cdots  \hspace{0.2cm} \gamma'_1\\
\hspace{0.2cm} \gamma'_{2k} \hspace{0.2cm} \cdots \hspace{0.2cm} \gamma'_2
\end{pmatrix}
$$
be the corresponding distinguished $u$-symbols. Distinguishedness ,and the inequality $l>2l'$ imply that $\gamma'_1=l-l'+2k$, and $\gamma_1=\xi_1+2k$.

For $r\in\{1,\ldots,2k+1\}$, there exist $P$, $Q\subset\{1,\ldots,r\}$ such that 
$$
\gamma'_1+\cdots+\gamma'_r = (l-l'+2k) + \sum_{P}\xi'_i+2(k-i)+\sum_{Q}\eta'_i+2(k-i)+1.
$$
The right side in the above inequality is smaller than
$$
(\xi_1+2k) + \sum_{P}\xi_{i+1}+2(k-i-1)+\sum_{Q}\eta_i+2(k-i)+1,
$$
by item 2 of Lemma \ref{TechnicalLemmaSOandUpairs}. The last sum is smaller than $\gamma_1+\cdots+\gamma_r$ (because of distinguishedness). This means that $\gamma'\leq\gamma$, and as in the proof of Proposition \ref{springer_ordre_so1}, this is equivalent to $\lambda((l-l')\cup\xi',\eta')\leq \lambda(\xi,\eta)$.
  
For unicity, consider the $u$-symbol corresponding to $((l-l')\cup\xi',\eta')$. We prove that all the other $u$-symbols in its similarity class correspond to bipartitions which do not belong to $\Theta(\xi',\eta')$. This implies that any $(\xi,\eta)$ in $\Theta(\xi',\eta')$ verifies $\lambda(\xi,\eta)\neq\lambda((l-l')\cup\xi',\eta')$. Unicity of the minimal representation follows.

Let $(\xi,\eta)$ be the partition whose $u$-symbol is similar (and different from) the one of $((l-l')\cup\xi',\eta')$. In this situation, there are three possibilities :

(1) $\lvert \xi\lvert > \lvert \xi' \lvert + l-l'$. Since the biggest entry in $\xi$ is $l-l'$, in order to obtain its Young diagram we would need to add at least two boxes in a same column to the diagram of $\xi'$. This contradicts $\xi'\preceq\xi$. 

(2) $\lvert \xi\lvert < \lvert \xi' \lvert + l-l'$. This implies that $\lvert \eta\lvert > \lvert \eta'\lvert$, so that the Young diagram of $\eta$ would not contain the one of $\eta'$. This contradicts $\eta\preceq\eta'$.

(3) $\lvert \xi\lvert = \lvert \xi' \lvert + l-l'$. Since the partitions are different, we would find certain $t$ such that $\xi'_t<\xi_{t+1}$. As in item (1), this contradicts $\xi'\preceq\xi$.

In any case, we see that $(\xi,\eta)$ does not belong to $\Theta(\xi',\eta')$. The assertion concerning the maximal representation has a similar proof.
\end{proof}

\subsection{Unitary pairs}
Let $\mu$ be a partition of $n$, and denote by the same letter the corresponding irreducible character of $\mathfrak{S}_n$. Define 
$$
R_\mu= \frac{1}{n!}\sum_{\sigma \in \mathfrak{S}_n} \mu(\sigma)R_{\mathbf{T}_\sigma}^{\mathbf{U}_n}(1).
$$
These central characters provide all the unipotent representations of $\textbf{U}_n(q)$ up to the sign
$$
\varepsilon_\mu=(-1)^{\sum_{i=1}^k\binom{\mu_i}{2}+\binom{m(m-1)}{2}}.
$$
\begin{prop}\label{UnipotentRepresentationUnitary}
The characters of unipotent irreducible representations of $\mathbf{U}_n(q)$ are given by 
$\varepsilon_\mu R_\mu$, for different partitions $\mu$ of $n$. 
\end{prop}
In what follows we need to use some combinatorics related to partitions. Let $D(\mu)$ denote the Young diagram of $\mu$. We call \emph{rim} of $\mu$, the boxes belonging to $\{(i,j)\in D(\mu)\rvert (i,j+1)\notin D(\mu) \mbox{ or }(i+1,j)\notin D(\mu)\}$. We call \emph{rim $2$-hook} of $\mu$ a pair $\{(i,j),(i,j+1)\}$ or $\{(i,j),(i+1,j)\}$ of elements of the rim of $\mu$, such that we obtain the diagram of a partition after removing these elements from $D(\mu)$. Finally, the $2$-core of $\mu$ is the partition obtained by removing as many rim $2$-hooks as possible from the diagram of $\mu$.

Let again $\mu=(\mu_1,\ldots,\mu_r)$ be a partition of $m$, and let $t\geq r$ be an integer. We call \emph{$t$-set of $\beta$-numbers} associated to $\mu$, the set $\beta^t_\mu=\{\beta_i\}$ where $\beta_i=\mu_i+t-i$. Conversely, to each decreasing sequence $\beta=\{\beta_1,\ldots,\beta_t\}$ of positive integers, we can associate a partition $\mu$, defined by $\mu_i=\beta_i+t-i$.

Set $\beta^t_\mu(0)=\{\beta_i/2\rvert \beta_i \mbox{ is even}\}$ and $\beta^t_\mu(1)=\{(\beta_i-1)/2\rvert \beta_i \mbox{ is odd}\}$. Let $\mu(0)$ and $\mu(1)$ be the partitions associated to the sets of $\beta$-numbers $\beta^t_\mu(0)$ and $\beta^t_\mu(1)$. Then, the pair $(\mu(0),\mu(1))$ depends only on the congruence class $\overline{t}$ of $t$ modulo 2 . We call $\mu(0)$ and $\mu(1)$ the \emph{$2$-quotients of parameter $\overline{t}$} of $\mu$ .
  
We have a bijection between the category $\mathscr{R}(\mathbf{U}_m(q))_k$, of complex representations spanned by $\Irr(\mathbf{U}_m(q))_k$, and the set of irreducible representations of $W_{\frac{1}{2}(m-k(k+1)/2)}$. This bijection allows us to describe explicitely the characters in this Harish-Chandra series.
\begin{prop}\cite[Appendice, proposition p. 224]{Fong-Srinivasan}\label{Parametrization_HCseries_Unitary}
\begin{itemize}
\item[1.] The unique cuspidal unipotent representation $\lambda_k$ of the unitary group $\mathbf{U}_{(k^2+k)/2}(q)$ is $\varepsilon_{\tau_k}R_{\tau_k}$, where $\tau_k$ is the $k$-th 2-core $\tau_k=(k,\ldots,1)$.
\item[2.] For $n\geq 1/2(k^2+k)$ the irreducible characters of $\mathscr{R}(\mathbf{U}_n(q))_k$ are $\varepsilon_\mu R_{\mu}$ where $\mu$ is a partition of $n$ of $2$-core $\tau_k$. This character is related to the bipartition $(\mu(0),\mu(1))$ (where $\mu(0)$ and $\mu(1)$ are the $2$-quotients of parameter $1$ of $\mu$) under the bijection given in Theorem \ref{How-Leh}.   
\end{itemize}
\end{prop}
This theorem tell us that for a fixed $k$ there is a bijection between the bipartitions of $\frac{1}{2}(n-k(k+1)/2)$, the representations in the Harish-Chandra series $\mathscr{R}(\mathbf{U}_n(q))_k$, and the partitions of $n$ with $2$-core $\tau_k$.  
\begin{prop}\cite[Lemma 5.8]{AMR}\label{betaset_parity_invariance}
Let $\mu'$ be the partition obtained from $\mu$ by removing a $2$-rimhook, then the $t$-set of $\beta$-numbers of $\mu'$ is equal to $\{\beta_1,\ldots,\beta_{j-1},\beta_j-2,\beta_{j+1},\ldots,\beta_t\}$, for a certain $j\leq t$. In particular, the $\beta$-sets of a partition and its $2$-core have the same number of even (resp. odd) elements. 
\end{prop}
\begin{rmk}\label{OddandEvenElementsinBetaSet}
Let $\beta_k$ be the $t$-set of $\beta$-numbers of $\tau_k$, and $t_0=\lvert\beta_k(0)\rvert$ and $t_1=\lvert\beta_k(1)\rvert$ be its number of even and odd parts respectively. These two numbers depend on the parity of $k$ and $t$. For instance, let $t$ be odd. If $k$ is even, then $t_0=(t+k+1)/2$ and $t_1=(t-k-1)/2$; and if $k$ is odd then $t_0=(t-k)/2$ and  $t_1=(t+k)/2$.
\end{rmk}
\begin{prop}\label{2_quotient_fixed_length}
Let $\mu$ and $\mu'$ be two bipartitions of the same integer, and $t$ be odd. Then $\mu$ and $\mu'$ have the same $2$-core if and only if $|\beta^t_\mu(0)|=|\beta^t_{\mu'}(0)|$ and $|\beta^t_{\mu}(1)|=|\beta^t_{\mu'}(1)|$. 
\end{prop} 
\begin{proof}
Suppose that $\mu$ and $\mu'$ have the same $2$-core. Proposition \ref{betaset_parity_invariance} says that the number of even (resp. odd) elements in the $t$-sets of $\beta$-numbers of $\mu$ and $\mu'$ equal the number of even (resp. odd) elements in the $t$-set of $\beta$-numbers of the common $2$-core, so we have $\lvert\beta^t_\mu(0)\rvert=\lvert\beta^t_{\mu'}(0)\rvert$ and $\lvert\beta^t_{\mu}(1)\rvert=\lvert\beta^t_{\mu'}(1)\rvert$.
 	 	
If $\mu$ and $\mu'$ have different 2-cores $\tau_k$ and $\tau_{k'}$, assuming that $k<k'$ we have 4 cases depending on the parity of $k$ and $k'$. For instance, if they're both odd then $\lvert\beta_k(1)\rvert=(t+k)/2<(t+k')/2=\lvert\beta_{k'}(1)\rvert$ and $\lvert\beta_{k'}(0)\rvert=(t-k')/2<(t-k)/2=\lvert\beta_k(0)\rvert$, so by Proposition \ref{betaset_parity_invariance}, $\lvert\beta^t_{\mu'}(0)\rvert<\lvert\beta^t_\mu(0)\rvert$ and $\lvert\beta^t_\mu(1)\rvert<\lvert\beta^t_{\mu'}(1)\rvert$. The other 3 cases are analogous. 
\end{proof}
For the unitary dual pair $(\mathbf{U}_m(q),\mathbf{U}_{m'}(q))$, the Howe correspondence between unipotent Harish-Chandra series induces a correspondence for the pair $(W_{\frac{1}{2}(m-k(k+1)/2)}, W_{\frac{1}{2}(m'-\theta(k)(\theta(k)+1)/2)})$ of type $\mathbf{B}$ Weyl groups, where $\theta(k)$ is equal to $k-1$ or $k+1$ (see Theorem \ref{CuspidalUnipotentHowe}). 
 
Let $l=\frac{1}{2}(m-k(k+1)/2)$ and $l'=\frac{1}{2}(m'-\theta(k)(\theta(k)+1)/2)$. The correspondence for the pair $(W_l,W_{l'})$  is given by the representations $\Omega_{l,l'}$ introduced in Theorem \ref{ReductionWeilUnipotent}.  
\begin{defi}
Fix a bipartition  $(\xi',\eta')$ of $l'$. We denote by $\Theta(\xi',\eta')$ the set of all bipartitions $(\xi,\eta)$ of $l$, such that $\chi_{\xi,\eta}\otimes\chi_{\xi',\eta'}$ appears in $\Omega_{l,l'}$.
\end{defi}
Proposition \ref{Parametrization_HCseries_Unitary} provides a bijection between $\Irr(W_l)$ and $\mathscr{R}(\mathbf{U}_m(q))_k$, sending the irreducible representation $\chi_{\xi,\eta}$ to the unipotent character $\varepsilon_\mu R_{\mu}$, where $\mu=\mu_k(\xi,\eta)$ is the partition of $n$ with $\tau_k$ as 2-core and $(\xi,\eta)$ as 2-quotient of parameter 1.
\begin{defi}\label{def_min_u}
Fix a bipartition  $(\xi',\eta')$ of $l'$.
\begin{itemize}
\item[1.] We define a partial order on $\Theta(\xi',\eta')$ by $(\xi,\eta)\preceq(\xi^{''},\eta^{''})$ if and only if $\mu_{k}(\xi,\eta)\leq \mu_{k}(\xi^{''},\eta^{''})$.
\item[2.] A bipartition in $\Theta(\xi',\eta')$ is \emph{minimal} (resp. \emph{maximal}), if it is so for the order just defined.
\end{itemize}
\end{defi}
We show in the following sections that $\Theta(\xi',\eta')$ admits a unique minimal (resp. maximal) representation.

The explicit form of $\Omega_{l,l'}$ depends on the parity of $k$. Theorem \ref{ReductionWeilUnipotent} provides two cases, we study these separately. As for symplectic-orthogonal pairs, we suppose the pair to be in the stable range (with $\mathbf{U}_{m'}(q)$ the smaller), this condition implies that $l\geq 2l'$.
 
\subsubsection{First case}
Consider the pair $(\mathbf{U}_m(q),\mathbf{U}_{m'}(q))$ for $k$ odd or $k=\theta(k)=0$.  In these cases, the Howe correspondence is given by the representation $\Omega_{l,l'}$ given in Theorem \ref{ReductionWeilUnipotent}. Proposition \ref{Weyl_representation_tensor} allows us to write. 
$$
\Omega_{l,l'} = \sum_{r=0}^{\min(l,l')}\sum_{(\xi,\eta)\in\mathcal{P}_2(r)}\sum_{\xi',\eta'}\chi_{\xi',\eta}\otimes\chi_{\eta',\xi},
$$
the third sum being over partitions $\xi'$ and $\eta'$ of $l-\lvert\eta\rvert$ and $l'-\lvert\xi\rvert$,  such that $\xi\preceq\xi'$ and $\eta\preceq\eta'$. The proof of the following lemma is a consequence of the previous identity and Lemma \ref{TechnicalLemmaSOandUpairs}.
\begin{lem}\label{smallest_lexicographic_order_u1} 
Let $(\xi',\eta')$ be a bipartition of $l'$. 
\begin{itemize}
\item[1.] A bipartition $(\eta,\xi)$ belongs to $\Theta(\xi',\eta')$, if and only if $\eta'\preceq \eta$ and $\xi\preceq\xi'$.
\item[2.] The smallest element in $\Theta(\xi',\eta')$ for the natural order is $((l-l')\cup\eta',\xi')$.
\item[3.] The largest element in $\Theta(\xi',\eta')$ for the natural order is $((l-l'+\xi'_1+\eta'_1,\eta'_2,\cdots,\eta'_r),(\xi'_2,\cdots,\xi'_r))$. 
\end{itemize}
\end{lem}
\begin{thm}\label{order_phantom_characters_u1}
Let $(\xi',\eta')$ be a bipartition of $l'$. 
\begin{itemize}
\item[1.] The unique minimal representation in $\Theta(\xi',\eta')$ corresponds to the bipartition $((l-l')\cup\eta',\xi')$ of $l$.
\item[2.] The unique maximal representation in $\Theta(\xi',\eta')$ corresponds to the bipartition $((l-l'+\xi'_1+\eta'_1,\eta'_2,\cdots,\eta'_r),(\xi'_2,\cdots,\xi'_r))$ of $l$.
\end{itemize}
\end{thm}
\begin{proof}
Take $(\eta,\xi)$ in $\Theta(\xi',\eta')$, let $\mu=\mu_k(\eta,\xi)$, and $\mu_{\mathrm{m}}=\mu_k((l-l')\cup\eta',\xi')$. We intent prove that $\mu_{\mathrm{m}}\leq\mu$. 

Let $t'_0$ (resp. $t'_1$) be the number of parts of $\xi'$ (resp. $\eta'$). By adding zeroes (if necessary) we can assume that $t'=t'_0+t'_1$ is odd, $\eta$ has $t'_1+1$ parts, and $\xi$ has $t'_0$ parts. Therefore, if we are to use $(\eta,\xi)$ as a 2-quotient of parameter 1, we must add a zero to either $\eta$ or $\xi$, so the total number of parts becomes $t=t'+2$ (odd).

Let $k'=\theta(k)$ be even and $k=k'-1$, or $k'$ be odd and $k=k'+1$. Using the formulas in Remark \ref{OddandEvenElementsinBetaSet}, the number of parts of $\eta=\mu(0)$ (resp. $(l-l')\cup\eta'=\mu_{\mathrm{m}}(0)$) should equal $t_0=t'_1+2$, and those of $\xi=\mu(1)$ should be $t_1=t'_0$. Therefore we must add a zero to $\eta$ (resp. $(l-l')\cup\eta'$).

Let $\beta$ (resp. $\beta_{\mathrm{m}}$) be the $t$-set of $\beta$-numbers of $\mu$ (resp. $\mu_{\mathrm{m}}$), a simple calculation shows that
\begin{align*}
\beta=\{2(\xi_1+t'_0)-1,\ldots, 2\xi_{t'_0}+1,2(\eta_1+t'_1+1),\ldots,2(\eta_{t'_1+1}+1),0\},\\
\beta_{\mathrm{m}}=\{2(\xi'_1+t'_0)-1,\ldots, 2\xi'_{t'_0}+1,2(l-l'+t'_1+1),\ldots,2(\eta'_{t'_1}+1),0\}.
\end{align*}
Let $\beta_1>\cdots>\beta_t$ (resp. $\beta_{\mathrm{m},1}>\cdots>\beta_{\mathrm{m},t}$) the elements of $\beta$ (resp. $\beta_{\mathrm{m}}$) after reordering. It is easy to see that $\mu_{\mathrm{m}}\leq\mu$ if and only if $\beta_1+\cdots+\beta_k\leq\beta_{\mathrm{m},1}+\cdots+\beta_{\mathrm{m},k}$, for $k=1,\ldots, t$.

Let $k\in\{1,\ldots,t\}$. There exist $r\in\{1,\ldots,t'_1\}$, and $s\in\{1,\ldots,t'_0\}$ such that $\beta_{\mathrm{m},1}+\cdots+\beta_{\mathrm{m},k}$ equals
$$
2(l-l'+t'_1+1) + \sum_{i=1}^r 2(\eta'_i+t'_1-i+1)+\sum_{i=1}^s 2(\xi'_i+t'_0-i)+1.
$$
The right side in the above inequality is smaller than
$$
2(\eta_1+t'_1+1) + \sum_{i=1}^r 2(\eta_{i+1}+t'_1-i+1)+\sum_{i=1}^s 2(\xi_i+t'_0-i)+1,
$$
by item 1 of Lemma \ref{TechnicalLemmaSOandUpairs}, and the last sum is smaller than $\beta_1+\cdots+\beta_k$. The case $k'$ odd and $k=k'+1$, or $k'$ even and $k=k'-1$ can be treated in a similar way. The same arguments yield item 2.
\end{proof}
	
\subsubsection{Second case}
Consider the pair $(\mathbf{U}_m(q),\mathbf{U}_{m'}(q))$ for $k$ even and different from zero. In these cases, the Howe correspondence is given by the representation $\Omega_{l,l'}$ given in Theorem \ref{ReductionWeilUnipotent}. Once again, Proposition \ref{Weyl_representation_tensor} allows us to write. 
$$
\Omega_{l,l'} = \sum_{r=0}^{\min(l,l')}\sum_{(\xi,\zeta)\in\mathcal{P}_2(r)}\sum_{\eta,\eta'}\chi_{\xi,\eta}\otimes\chi_{\eta',\xi},
$$
where the third sum is over partitions $\eta$ and $\eta'$, of $l-\lvert\xi\rvert$ and $l'-\lvert\xi\rvert$, such that $\zeta\preceq\eta$ and $\zeta\preceq\eta'$.
\begin{lem}\label{smallest_lexicographic_order_u2}
Let $(\xi',\eta')$ be a bipartition of $l'$.  
\begin{itemize}	
\item[1.] A bipartition $(\eta,\xi)$ of $l$ belongs to $\Theta(\xi',\eta')$, if and only if, $\xi$ and $\xi'$ have a common predecesor for $\preceq$, and $\eta'=\eta$.
\item[2.] The smallest element in $\Theta(\xi',\eta')$ for the natural order is $(\eta',(l-l')\cup\xi')$.
\item[3.] The largest element in $\Theta(\xi',\eta')$ for the natural order is $(\eta',(l-l'+\xi'_1+\xi'_2,\xi'_3,\ldots,\xi'_r))$.
\end{itemize}
\end{lem}
\begin{prop}\label{prop_order_phantom_characters_u2}
Let $(\eta,\xi)$ and $(\eta,\xi')$ be two bipartitions of $l$. If $(\eta,\xi)\leq(\eta,\xi')$ then $\mu_k(\eta,\xi)\leq\mu_k(\eta,\xi')$.
\end{prop}
\begin{proof}
Call $t_0$ the number of parts of $\eta$, let $\mu=\mu_k(\eta,\xi)$, and $\mu'=\mu_k(\eta,\xi')$. As $\mu$ and $\mu'$ have the same 2-core, we can suppose $\xi$ and $\xi'$ to have the same number of parts $t_1$, so that the $t$-sets of $\beta$-numbers of $\mu$ and $\mu'$ are
$$
\beta=\{2(\eta_1+t_0-1),\ldots,2\eta_{t_0},2(\xi_1+t_1)-1,\ldots,2\xi_{t_1}+1\},
$$
and
$$
\beta'=\{2(\eta_1+t_0-1),\ldots,2\eta_{t_0},2(\xi'_1+t_1)-1,\ldots,2\xi'_{t_1}+1\}.
$$	
Suppose that, after ordering, the elements of $\beta$ (resp. $\beta'$) are $\beta_1>\cdots>\beta_t$, (resp. $\beta'_1>\cdots>\beta'_t$) for $t=t_0+t_1$. It is easy to see that $\mu\leq\mu'$ if and only if $\beta_1+\cdots+\beta_k\leq\beta_{\mathrm{m},1}+\cdots+\beta_{\mathrm{m},k}$, for $k=1,\ldots, t$.

Let $k\in\{1,\ldots,t\}$. We can find non negatives integers $r\in\{1,\ldots,t_0\}$ and $s\in\{1,\ldots,t_1\}$, such that $\beta_1+\cdots+\beta_k$ is equal to
$$
\sum_{i=1}^r 2(\eta_i+t_0-i)+\sum_{i=1}^s 2(\xi_i+t_1-i)+1.
$$
The condition $(\eta,\xi)\leq(\eta,\xi')$ (i.e.\ $\xi\leq\xi'$) implies that the previous sum is smaller or equal than 
$$
\sum_{i=1}^r 2(\eta_i+t_0-i)+\sum_{i=1}^s 2(\xi'_i+t_1-i)+1,
$$
which in turn is smaller than $\beta'_1+\cdots+\beta'_k$. 
\end{proof}
This proposition tells us that the map sending $(\eta,\xi)$ to $\mu_k(\eta,\xi)$ (for $k$ fixed) is increasing when restricted to $\Theta(\xi',\eta')$. These remarks and Lemma \ref{smallest_lexicographic_order_u2} imply the following :
\begin{thm}\label{order_phantom_characters_u2}
Let $(\xi',\eta')$ be a bipartition of $l'$. 
\begin{itemize}
\item[1.] The unique minimal representation in $\Theta(\xi',\eta')$ is given by the bipartition $(\eta',(l-l')\cup\xi')$.
\item[2.] The unique maximal representation in $\Theta(\lambda,\mu)$ is given by the bipartition $(\eta',(l-l'+\xi'_1+\xi'_2,\xi'_3,\ldots,\xi'_r))$. 
\end{itemize}

\end{thm}
\subsubsection{Results in terms of partitions}
In the previous two sections, we expressed the Howe correspondence between Harish-Chandra series as a correspondence between bipartitions (the so called ``2-quotients''). Moreover, we were able to exhibit extremal representations (seen as bipartitions) issued from this correspondence. Since Proposition \ref{Parametrization_HCseries_Unitary} also provides a parametrization of Harish-Chandra series in terms of partitions having the same 2-core, it is then natural to ask if these results can be expressed in terms of partitions.

In this section we provide a satisfactory answer for minimal representations whenever $\theta(k)=k+1$. As in the previous section we assume this pair to be in the stable range (with $\U_{m'}(q)$ smaller).
\begin{thm}
Let $(\U_m(q),\U_{m'}(q))$ be a dual pair. Then, the minimal representation corresponding to $\epsilon_{\mu'} R_{\mu'}$ is $\epsilon_{(m-m')\cup\mu'} R_{(m-m')\cup\mu'}$, provided that $\theta(k)=k+1$.
\end{thm}
\begin{proof}
Proposition \ref{Parametrization_HCseries_Unitary} tells us that the partition $\mu'$ of $m'$ has $\tau_{\theta(k)}$ as its $2$-core. Let $(\xi',\eta')$ be its $2$-quotient of parameter 1. Adding zeroes if necessary, we can suppose that the number of parts $t'_0$ of $\xi'$, and $t'_1$ of $\eta'$, have an odd sum $t'$, so that the $t$-set of $\beta$-numbers of $\mu'$ has an odd cardinal. An easy calculation shows that this set is : 
$$
\beta'=\{2(\xi'_1+t'_0-1),\ldots,2\xi'_{t'_0},2(\eta'_1+t'_1)-1,\ldots,2\eta'_{t'_1}+1\},
$$  
where the elements are not necessarily in decreasing order. 

We take $t=t'+2$ (odd) for computing the $t$-set of $\beta$-numbers of the partition $\mu_{\mathrm{min}}$ corresponding to the minimal bipartition. The hypothesis $\theta(k)=k+1$, together with Remark \ref{OddandEvenElementsinBetaSet} imply that :

(1) If $k$ is even, then $t_0=t'_1+1$, and $t_1=t'_0+1$. Moreover, in this case the minimal bipartition is $(\eta',(l-l')\cup\xi')$. Therefore, the $t$-set of $\beta$-numbers of $\mu_{\mathrm{min}}$ is
$$
\beta_{\mathrm{min}}=\{2(l-l'+t'_0)+1\}\cup\{x+1|x\in\beta'\}\cup\{0\}.
$$ 

(2) If $k$ is odd, then $t_0=t'_1+2$, and $t_1=t'_0$. In this case the minimal bipartition corresponds to $((l-l')\cup\eta',\xi')$, so the $t$-set of $\beta$-numbers of $\mu_{\mathrm{min}}$ is
$$
\beta_{\mathrm{min}}=\{2(l'-l+t_1+1)\}\cup\{x+1|x\in\beta'\}\cup\{0\}.
$$ 
Recalling that $l\geq 2l'$, a final simple calculation shows that $\mu_{\mathrm{min}}$ equals $(m-m')\cup\mu'$.
\end{proof}